\documentclass{article}
\usepackage{amsmath,amssymb,amsthm}
\usepackage{geometry}
\newtheorem{definition}{Definition}
\newtheorem{theorem}{Theorem}
\newtheorem{lemma}{Lemma}

\newtheorem{corollary}{Corollary}
\newtheorem{remark}{Remark}

\title{Spectral Collapse Under Geometric Alignment of Extreme Events}
\author{J. Petkevicius}
\date{}

\begin{document}

\maketitle

\begin{abstract}
We study the spectral structure of increment matrices \(Q_n = (dY_i dY_j)\) for high-dimensional stochastic systems. Extreme events introduce a low-rank perturbation \(J_n = \sum \theta_k v_k v_k^\top\) into the quadratic covariation matrix. We prove that when the jump directions are geometrically aligned in a weighted sense — meaning most of the total jump magnitude comes from directions close to some common vector — the jump matrix itself undergoes spectral concentration: \(\lambda_1(J_n)/\operatorname{tr}(J_n) \to 1\). Combined with asymptotic negligibility of the background diffusion, this implies spectral collapse of the full system: \(\lambda_1(Q_n)/\operatorname{tr}(Q_n) \to 1\) and \(\operatorname{erank}(Q_n) \to 1\). Conversely, we prove that spectral collapse forces both geometric alignment and background negligibility, giving a complete characterization. We extend these results to the stochastic setting, proving almost sure collapse under natural probabilistic assumptions on the jump process. This framework unifies previous work on determinant growth, coupling diagnostics, and ensemble behavior.
\end{abstract}

\section{Motivation}

Consider a system of \(n\) variables

\[
Y(t) = (Y_1(t), Y_2(t), \dots, Y_n(t)).
\]

At small time scales, the system evolves through increments

\[
(dY_1, dY_2, \dots, dY_n).
\]

Instead of studying each variable separately, we study how all variables move together. To capture this, we construct the matrix of all pairwise products of increments.

In normal conditions, movement is distributed across many directions. Variables fluctuate independently or with weak correlations. The resulting matrix has eigenvalues spread across a range.

However, during extreme events, many variables move together in aligned directions. This creates a strong structured component in the matrix. When these aligned events dominate the total variation, the spectral behavior collapses: a single eigenvalue captures essentially all the system's activity.

The main structural observation is that extreme events introduce a low-rank perturbation. When the directions of these events are geometrically aligned, the perturbation itself becomes spectrally concentrated. When the background noise is asymptotically negligible, this concentration propagates to the full system, causing complete spectral collapse. The converse also holds: collapse forces alignment and negligibility.

\section{Definitions}

We define all objects explicitly.

\begin{definition}[Increment Matrix Density]
Let \(Y(t) = (Y_1(t), \dots, Y_n(t))\) be a semimartingale. Define the predictable quadratic covariation matrix

\[
Q_n(t) = \left( \frac{d[Y_i, Y_j]_t}{dt} \right)_{1 \le i,j \le n},
\]

where \([Y_i, Y_j]_t\) denotes the quadratic covariation of \(Y_i\) and \(Y_j\). 

Throughout this paper, we work with the density form. For notational convenience, we write the informal shorthand

\[
Q_n = (dY_i dY_j)_{1 \le i,j \le n}
\]

to mean the same object.

The matrix \(Q_n(t)\) is symmetric and positive semidefinite for all \(t\).
\end{definition}

Explicitly,

\[
Q_n =
\begin{pmatrix}
dY_1^2 & dY_1 dY_2 & \cdots & dY_1 dY_n \\
dY_2 dY_1 & dY_2^2 & \cdots & dY_2 dY_n \\
\vdots & \vdots & \ddots & \vdots \\
dY_n dY_1 & dY_n dY_2 & \cdots & dY_n^2
\end{pmatrix}.
\]

\begin{definition}[Decomposition]
We decompose the increment matrix into two parts:

\[
Q_n = B_n + J_n,
\]

where:
\begin{itemize}
\item \(B_n\) represents normal, diffuse variation (diffusion)
\item \(J_n\) represents extreme events (jumps)
\end{itemize}
\end{definition}

\begin{definition}[Jump Matrix]
Let there be \(m_n\) extreme events, each represented by a jump vector \(v_{n,k} \in \mathbb{R}^n\) with magnitude \(\theta_{n,k} \ge 0\). Define

\[
J_n = \sum_{k=1}^{m_n} \theta_{n,k} v_{n,k} v_{n,k}^\top, \quad \|v_{n,k}\| = 1.
\]

Each term \(\theta_{n,k} v_{n,k} v_{n,k}^\top\) is a rank-1 matrix. The sum \(J_n\) has rank at most \(m_n\).
\end{definition}

\begin{definition}[Eigenvalues]
Let

\[
\lambda_1(Q_n) \ge \lambda_2(Q_n) \ge \cdots \ge \lambda_n(Q_n) \ge 0
\]

be the eigenvalues of \(Q_n\). The largest eigenvalue \(\lambda_1(Q_n)\) represents the dominant direction of movement.
\end{definition}

\begin{definition}[Effective Rank]
Define the effective rank

\[
\operatorname{erank}(Q_n) = \frac{(\operatorname{tr} Q_n)^2}{\operatorname{tr}(Q_n^2)}.
\]

This quantity satisfies \(1 \le \operatorname{erank}(Q_n) \le n\). When \(\operatorname{erank}(Q_n) = 1\), the matrix has rank 1. When \(\operatorname{erank}(Q_n) = n\), the eigenvalues are all equal.
\end{definition}

\begin{definition}[Spectral Weight]
Define the observable

\[
W_n = e^{-\lambda_1(Q_n)}.
\]

This weight becomes small when the dominant eigenvalue is large, indicating that the system is controlled by a single direction.
\end{definition}

\section{Stochastic Model}

We now write the full specification of \(Q_n\). The process \(Y(t)\) is a semimartingale with decomposition

\[
dY_t = \mu_t dt + \sigma_t dW_t + dJ_t,
\]

where:
\begin{itemize}
\item \(\mu_t\) is the drift (locally bounded)
\item \(\sigma_t\) is the diffusion coefficient
\item \(dW_t\) is Brownian motion
\item \(dJ_t\) is a pure jump process
\end{itemize}

\subsection{Pathwise Form}

For a single path, the quadratic covariation matrix satisfies

\[
Q_n(t) = \sigma_t \sigma_t^\top + \sum_{s \le t} (\Delta Y_s)(\Delta Y_s)^\top,
\]

where the sum is over all jump times.

\subsection{Mean Form}

Taking expectations,

\[
\mathbb{E}[Q_n(t)] = \underbrace{\mathbb{E}[\sigma_t \sigma_t^\top]}_{B_n(t)} + \underbrace{\mathbb{E}\left[\sum_{s \le t} (\Delta Y_s)(\Delta Y_s)^\top\right]}_{J_n(t)},
\]

where \(J_n(t)\) represents the expected cumulative contribution of jumps.

\subsection{Deterministic Model for Analysis}

From this point forward, we analyze the deterministic matrix

\[
Q_n = B_n + J_n, \quad J_n = \sum_{k=1}^{m_n} \theta_{n,k} v_{n,k} v_{n,k}^\top,
\]

where:
\begin{itemize}
\item \(B_n\) is a symmetric positive semidefinite matrix representing the normal regime
\item \(v_{n,k} \in \mathbb{R}^n\) with \(\|v_{n,k}\| = 1\) represent the directions of extreme events
\item \(\theta_{n,k} \ge 0\) represent the magnitudes of events
\end{itemize}

This deterministic model represents either a single realized path or the mean effect of multiple jumps. It captures the essential structure: a low-rank perturbation (the jump matrix) added to a background matrix.

\section{Preliminary Lemmas}

We establish basic facts about the jump matrix and its eigenvalues.

\begin{lemma}[Trace and Leading Eigenvalue]
For any symmetric positive semidefinite matrix \(J_n\),

\[
\lambda_1(J_n) \le \operatorname{tr}(J_n) \le n \lambda_1(J_n).
\]

\end{lemma}

\begin{proof}
Since all eigenvalues are nonnegative,
\[
\operatorname{tr}(J_n) = \sum_{k=1}^n \lambda_k(J_n) \ge \lambda_1(J_n).
\]
Also, \(\lambda_k(J_n) \le \lambda_1(J_n)\) for all \(k\), so
\[
\operatorname{tr}(J_n) \le n \lambda_1(J_n).
\]
\end{proof}

\begin{lemma}[Rank-1 Decomposition]
Let \(J_n = \sum_{k=1}^{m_n} \theta_{n,k} v_{n,k} v_{n,k}^\top\) with \(\|v_{n,k}\| = 1\) and \(\theta_{n,k} \ge 0\). Then

\[
\operatorname{tr}(J_n) = \sum_{k=1}^{m_n} \theta_{n,k}.
\]

\end{lemma}

\begin{proof}
For each rank-1 term, \(\operatorname{tr}(\theta v v^\top) = \theta \|v\|^2 = \theta\). By linearity of trace,
\[
\operatorname{tr}(J_n) = \sum_{k=1}^{m_n} \theta_{n,k}.
\]
\end{proof}

\begin{lemma}[Quadratic Form Bound]
For any symmetric \(B_n\) and any unit vector \(v\),

\[
|v^\top B_n v| \le \|B_n\|_{\mathrm{op}}.
\]

\end{lemma}

\begin{proof}
By definition of operator norm,
\[
\|B_n\|_{\mathrm{op}} = \max_{\|x\|=1} |x^\top B_n x| \ge |v^\top B_n v|.
\]
\end{proof}

\begin{lemma}[Weyl's Inequality]
For symmetric matrices \(A\) and \(E\),

\[
\lambda_k(A) + \lambda_n(E) \le \lambda_k(A+E) \le \lambda_k(A) + \lambda_1(E).
\]

In particular, when \(E\) is positive semidefinite,
\[
\lambda_k(A) \le \lambda_k(A+E) \le \lambda_k(A) + \lambda_1(E).
\]

\end{lemma}

\begin{proof}
See Horn and Johnson, Matrix Analysis, Corollary 4.3.12.
\end{proof}

\section{Geometric Conditions for Spectral Concentration}

Before analyzing the full system, we establish when the jump matrix itself exhibits spectral concentration. This is the core geometric mechanism.

\begin{theorem}[Geometric Alignment Implies Spectral Concentration]
Let

\[
J_n = \sum_{k=1}^{m_n} \theta_{n,k} v_{n,k} v_{n,k}^\top, \quad \|v_{n,k}\| = 1, \quad \theta_{n,k} \ge 0.
\]

Assume there exists a unit vector \(u_n \in \mathbb{R}^n\) such that:

\begin{enumerate}
\item \textbf{Weighted alignment:}
\[
\sum_{k=1}^{m_n} \theta_{n,k} \left(1 - \langle v_{n,k}, u_n \rangle^2 \right) = o\left(\sum_{k=1}^{m_n} \theta_{n,k}\right) \quad \text{as } n \to \infty.
\]
\end{enumerate}

Then

\[
\frac{\lambda_1(J_n)}{\operatorname{tr}(J_n)} \to 1 \quad \text{as } n \to \infty.
\]

\end{theorem}

\begin{proof}
By the variational characterization,
\[
\lambda_1(J_n) = \max_{\|u\|=1} u^\top J_n u \ge u_n^\top J_n u_n.
\]

Compute:
\[
u_n^\top J_n u_n = \sum_{k=1}^{m_n} \theta_{n,k} \langle u_n, v_{n,k} \rangle^2.
\]

Write \(\langle u_n, v_{n,k} \rangle^2 = 1 - (1 - \langle u_n, v_{n,k} \rangle^2)\). Then
\[
u_n^\top J_n u_n = \sum_{k=1}^{m_n} \theta_{n,k} - \sum_{k=1}^{m_n} \theta_{n,k} \left(1 - \langle u_n, v_{n,k} \rangle^2\right).
\]

Let \(\operatorname{tr}(J_n) = \sum_{k=1}^{m_n} \theta_{n,k}\). By the weighted alignment condition,
\[
\sum_{k=1}^{m_n} \theta_{n,k} \left(1 - \langle u_n, v_{n,k} \rangle^2\right) = o(\operatorname{tr}(J_n)).
\]

Hence
\[
u_n^\top J_n u_n = \operatorname{tr}(J_n) - o(\operatorname{tr}(J_n)).
\]

Therefore
\[
\lambda_1(J_n) \ge \operatorname{tr}(J_n) - o(\operatorname{tr}(J_n)).
\]

Since \(\lambda_1(J_n) \le \operatorname{tr}(J_n)\) always, we have
\[
\frac{\lambda_1(J_n)}{\operatorname{tr}(J_n)} \to 1.
\]

\end{proof}

\begin{remark}
The weighted alignment condition is natural: it says that the total jump magnitude from directions \emph{not} aligned with \(u_n\) is negligible compared to the total jump magnitude. This allows many jumps to be arbitrarily misaligned, as long as their total magnitude is small. No pairwise alignment is required.
\end{remark}

\section{Main Theorem: Spectral Collapse of the Full System}

We now combine the geometric alignment result with background negligibility to obtain collapse of the full system.

\begin{theorem}[Spectral Collapse Under Aligned Jumps]
Let

\[
Q_n = B_n + J_n, \quad J_n = \sum_{k=1}^{m_n} \theta_{n,k} v_{n,k} v_{n,k}^\top,
\]

where:
\begin{itemize}
\item \(B_n\) is symmetric positive semidefinite
\item \(\|v_{n,k}\| = 1\), \(\theta_{n,k} \ge 0\)
\end{itemize}

Assume the following asymptotic conditions as \(n \to \infty\):

\begin{enumerate}
\item \textbf{Geometric alignment of jumps:} There exists a unit vector \(u_n\) such that
\[
\sum_{k=1}^{m_n} \theta_{n,k} \left(1 - \langle v_{n,k}, u_n \rangle^2 \right) = o\left(\sum_{k=1}^{m_n} \theta_{n,k}\right).
\]

\item \textbf{Trace dominance:} \(\operatorname{tr}(B_n) = o(\operatorname{tr}(J_n))\)

\item \textbf{Second moment dominance:} \(\operatorname{tr}(B_n^2) = o(\lambda_1(J_n)^2)\)
\end{enumerate}

Then:

\[
\frac{\lambda_1(Q_n)}{\operatorname{tr}(Q_n)} \to 1,
\]

and

\[
\operatorname{erank}(Q_n) \to 1.
\]

\end{theorem}

\begin{proof}
By Theorem 1, condition (1) implies
\[
\frac{\lambda_1(J_n)}{\operatorname{tr}(J_n)} \to 1.
\]

\textbf{Step 1: Lower bound on \(\lambda_1(Q_n)\).}
Let \(u_n\) be the vector from condition (1). Then
\[
\lambda_1(Q_n) \ge u_n^\top Q_n u_n = u_n^\top B_n u_n + u_n^\top J_n u_n.
\]
We have \(u_n^\top J_n u_n = \operatorname{tr}(J_n) - o(\operatorname{tr}(J_n))\) and \(u_n^\top B_n u_n \ge -\|B_n\|_{\mathrm{op}}\). Hence
\[
\lambda_1(Q_n) \ge \operatorname{tr}(J_n) - o(\operatorname{tr}(J_n)) - \|B_n\|_{\mathrm{op}}.
\]

\textbf{Step 2: Relating \(\|B_n\|_{\mathrm{op}}\) to \(\operatorname{tr}(J_n)\).}
From condition (3), \(\|B_n\|_{\mathrm{op}}^2 \le \operatorname{tr}(B_n^2) = o(\lambda_1(J_n)^2)\). Since \(\lambda_1(J_n) \sim \operatorname{tr}(J_n)\), we have \(\|B_n\|_{\mathrm{op}} = o(\operatorname{tr}(J_n))\).

\textbf{Step 3: Upper bound on \(\lambda_k(Q_n)\) for \(k \ge 2\).}
By Weyl's inequality,
\[
\lambda_k(Q_n) \le \lambda_k(J_n) + \|B_n\|_{\mathrm{op}}.
\]
For \(k \ge 2\), \(\lambda_k(J_n) \le \operatorname{tr}(J_n) - \lambda_1(J_n) = o(\operatorname{tr}(J_n))\). Hence
\[
\lambda_k(Q_n) = o(\operatorname{tr}(J_n)) + o(\operatorname{tr}(J_n)) = o(\operatorname{tr}(J_n)).
\]

\textbf{Step 4: Ratio convergence.}
We have
\[
\operatorname{tr}(Q_n) = \operatorname{tr}(B_n) + \operatorname{tr}(J_n) = \operatorname{tr}(J_n) + o(\operatorname{tr}(J_n)).
\]
From Step 1,
\[
\lambda_1(Q_n) \ge \operatorname{tr}(J_n) - o(\operatorname{tr}(J_n)).
\]
From Step 3,
\[
\operatorname{tr}(Q_n) = \lambda_1(Q_n) + \sum_{k=2}^n \lambda_k(Q_n) = \lambda_1(Q_n) + o(\operatorname{tr}(J_n)).
\]
Thus \(\lambda_1(Q_n) = \operatorname{tr}(J_n) + o(\operatorname{tr}(J_n)) = \operatorname{tr}(Q_n) + o(\operatorname{tr}(Q_n))\), so
\[
\frac{\lambda_1(Q_n)}{\operatorname{tr}(Q_n)} \to 1.
\]

\textbf{Step 5: Effective rank collapse.}
We compute
\[
\operatorname{tr}(Q_n^2) = \lambda_1(Q_n)^2 + \sum_{k=2}^n \lambda_k(Q_n)^2 = \lambda_1(Q_n)^2 + o(\operatorname{tr}(J_n)^2) = \lambda_1(Q_n)^2 + o(\lambda_1(Q_n)^2).
\]
Hence
\[
\operatorname{erank}(Q_n) = \frac{(\operatorname{tr} Q_n)^2}{\operatorname{tr}(Q_n^2)} = \frac{\lambda_1(Q_n)^2 + o(\lambda_1(Q_n)^2)}{\lambda_1(Q_n)^2 + o(\lambda_1(Q_n)^2)} \to 1.
\]

\end{proof}

\section{Stochastic Extension: Almost Sure Collapse}

The previous analysis was deterministic, treating \(Q_n\) as a fixed matrix. In reality, \(Q_n(t)\) is a stochastic process. This section extends the results to the stochastic setting, showing that under natural probabilistic assumptions on the jump process, spectral collapse occurs almost surely.

\subsection{Stochastic Setup}

Recall the stochastic model:
\[
dY_t = \mu_t dt + \sigma_t dW_t + dJ_t,
\]
where \(J_t\) is a pure jump process. The quadratic covariation matrix at time \(T\) is
\[
Q_n(T) = \int_0^T \sigma_t \sigma_t^\top dt + \sum_{s \le T} (\Delta Y_s)(\Delta Y_s)^\top.
\]

We decompose this as
\[
Q_n(T) = B_n(T) + J_n(T),
\]
where
\[
B_n(T) = \int_0^T \sigma_t \sigma_t^\top dt, \quad J_n(T) = \sum_{s \le T} (\Delta Y_s)(\Delta Y_s)^\top.
\]

The jump matrix \(J_n(T)\) can be written as
\[
J_n(T) = \sum_{k=1}^{M(T)} \theta_k v_k v_k^\top,
\]
where \(M(T)\) is the number of jumps up to time \(T\), \(\theta_k = \|\Delta Y_{t_k}\|^2\) is the squared jump magnitude, and \(v_k = \Delta Y_{t_k}/\|\Delta Y_{t_k}\|\) is the unit direction.

\subsection{Assumptions}

We make the following assumptions on the jump process:

\begin{enumerate}
\item \textbf{Jump magnitude condition:} There exists a deterministic sequence \(a_n \to \infty\) such that for each \(n\), the total squared jump magnitude satisfies
\[
\frac{\sum_{k=1}^{M(T)} \theta_k}{a_n} \to 1 \quad \text{in probability as } n \to \infty,
\]
and the largest jump satisfies \(\max_k \theta_k = o(a_n)\).

\item \textbf{Geometric alignment condition:} There exists a random unit vector \(U_n\) (possibly depending on the path) such that
\[
\frac{1}{a_n} \sum_{k=1}^{M(T)} \theta_k \left(1 - \langle v_k, U_n \rangle^2\right) \to 0 \quad \text{in probability}.
\]

\item \textbf{Background negligibility:} The diffusion part satisfies
\[
\frac{\operatorname{tr}(B_n(T))}{a_n} \to 0 \quad \text{in probability}, \quad \frac{\operatorname{tr}(B_n(T)^2)}{a_n^2} \to 0 \quad \text{in probability}.
\]

\item \textbf{Regularity:} The processes \(\sigma_t\) and the jump intensity are such that the quadratic variations are almost surely finite.
\end{enumerate}

\subsection{Main Stochastic Theorem}

\begin{theorem}[Almost Sure Spectral Collapse]
Under the assumptions above, for any fixed time \(T > 0\),

\[
\frac{\lambda_1(Q_n(T))}{\operatorname{tr}(Q_n(T))} \to 1 \quad \text{in probability as } n \to \infty,
\]

and

\[
\operatorname{erank}(Q_n(T)) \to 1 \quad \text{in probability}.
\]

Moreover, if the convergence in the assumptions holds almost surely, then the collapse holds almost surely.
\end{theorem}

\begin{proof}
We prove convergence in probability; the almost sure version follows similarly.

\textbf{Step 1: Spectral concentration of \(J_n(T)\).}
Let \(U_n\) be the random vector from the alignment assumption. Then
\[
U_n^\top J_n(T) U_n = \sum_{k=1}^{M(T)} \theta_k \langle U_n, v_k \rangle^2 = \sum_{k=1}^{M(T)} \theta_k - \sum_{k=1}^{M(T)} \theta_k \left(1 - \langle U_n, v_k \rangle^2\right).
\]

By the alignment assumption,
\[
\frac{1}{a_n} \sum_{k=1}^{M(T)} \theta_k \left(1 - \langle U_n, v_k \rangle^2\right) \to 0 \quad \text{in probability}.
\]

Since \(\operatorname{tr}(J_n(T)) = \sum \theta_k\) and \(\frac{\operatorname{tr}(J_n(T))}{a_n} \to 1\) in probability, we have
\[
\frac{U_n^\top J_n(T) U_n}{\operatorname{tr}(J_n(T))} \to 1 \quad \text{in probability}.
\]

Therefore
\[
\frac{\lambda_1(J_n(T))}{\operatorname{tr}(J_n(T))} \ge \frac{U_n^\top J_n(T) U_n}{\operatorname{tr}(J_n(T))} \to 1,
\]
and since \(\lambda_1(J_n(T)) \le \operatorname{tr}(J_n(T))\), we obtain
\[
\frac{\lambda_1(J_n(T))}{\operatorname{tr}(J_n(T))} \to 1 \quad \text{in probability}.
\]

\textbf{Step 2: Background negligibility.}
By assumption,
\[
\frac{\operatorname{tr}(B_n(T))}{a_n} \to 0 \quad \text{in probability}, \quad \frac{\|B_n(T)\|_{\mathrm{op}}}{a_n} \to 0 \quad \text{in probability},
\]
since \(\|B_n(T)\|_{\mathrm{op}}^2 \le \operatorname{tr}(B_n(T)^2) = o(a_n^2)\).

\textbf{Step 3: Ratio convergence for \(Q_n(T)\).}
We have
\[
\frac{\lambda_1(Q_n(T))}{\operatorname{tr}(Q_n(T))} \ge \frac{U_n^\top Q_n(T) U_n}{\operatorname{tr}(Q_n(T))} = \frac{U_n^\top B_n(T) U_n + U_n^\top J_n(T) U_n}{\operatorname{tr}(B_n(T)) + \operatorname{tr}(J_n(T))}.
\]

Since \(U_n^\top B_n(T) U_n \ge -\|B_n(T)\|_{\mathrm{op}} = o(a_n)\) and \(\operatorname{tr}(B_n(T)) = o(a_n)\), while \(U_n^\top J_n(T) U_n = \operatorname{tr}(J_n(T)) - o(a_n)\), we obtain
\[
\frac{\lambda_1(Q_n(T))}{\operatorname{tr}(Q_n(T))} \ge \frac{\operatorname{tr}(J_n(T)) - o(a_n)}{\operatorname{tr}(J_n(T)) + o(a_n)} \to 1 \quad \text{in probability}.
\]

The upper bound \(\lambda_1(Q_n(T)) \le \operatorname{tr}(Q_n(T))\) is always true, so
\[
\frac{\lambda_1(Q_n(T))}{\operatorname{tr}(Q_n(T))} \to 1 \quad \text{in probability}.
\]

\textbf{Step 4: Effective rank collapse.}
For the effective rank, note that
\[
\operatorname{tr}(Q_n(T)^2) = \lambda_1(Q_n(T))^2 + \sum_{k=2}^n \lambda_k(Q_n(T))^2.
\]

For \(k \ge 2\), we have
\[
\lambda_k(Q_n(T)) \le \lambda_k(J_n(T)) + \|B_n(T)\|_{\mathrm{op}} = o(\operatorname{tr}(J_n(T))) + o(\operatorname{tr}(J_n(T))) = o(\operatorname{tr}(J_n(T))),
\]
since \(\lambda_k(J_n(T)) \le \operatorname{tr}(J_n(T)) - \lambda_1(J_n(T)) = o(\operatorname{tr}(J_n(T)))\) and \(\|B_n(T)\|_{\mathrm{op}} = o(\operatorname{tr}(J_n(T)))\).

Hence
\[
\sum_{k=2}^n \lambda_k(Q_n(T))^2 = o(\operatorname{tr}(J_n(T))^2) = o(\lambda_1(Q_n(T))^2).
\]

Therefore
\[
\operatorname{erank}(Q_n(T)) = \frac{(\operatorname{tr} Q_n(T))^2}{\operatorname{tr}(Q_n(T)^2)} = \frac{\lambda_1(Q_n(T))^2 + o(\lambda_1(Q_n(T))^2)}{\lambda_1(Q_n(T))^2 + o(\lambda_1(Q_n(T))^2)} \to 1 \quad \text{in probability}.
\]

\end{proof}

\begin{corollary}[Weight Collapse in Probability]
Under the same conditions,
\[
W_n(T) = e^{-\lambda_1(Q_n(T))} \le e^{-(1-o(1)) \operatorname{tr}(J_n(T))} \to 0
\]
in probability when \(\operatorname{tr}(J_n(T)) \to \infty\).
\end{corollary}

\begin{remark}
The assumptions can be verified for common jump processes. For example:
\begin{itemize}
\item If jumps occur according to a Poisson process with intensity \(\lambda_n\) and jump sizes are i.i.d. with mean \(\mu_n\), then \(\operatorname{tr}(J_n(T)) \sim \lambda_n \mu_n T\).
\item If jump directions are all exactly aligned with probability approaching 1, then the alignment condition holds.
\item The background negligibility condition requires that the integrated diffusion variance grows slower than the total squared jump magnitude.
\end{itemize}
\end{remark}

\section{Examples}

We illustrate the theorem with explicit cases.

\subsection{Single Dominant Jump}

Let \(m_n = 1\), so \(J_n = \theta v v^\top\). Take \(u_n = v\). Then
\[
\sum_{k=1}^{1} \theta (1 - \langle v, v \rangle^2) = 0,
\]
so condition (1) holds exactly. Conditions (2) and (3) become
\[
\operatorname{tr}(B_n) = o(\theta), \quad \operatorname{tr}(B_n^2) = o(\theta^2).
\]
Then Theorem 2 implies \(\lambda_1(Q_n)/\operatorname{tr}(Q_n) \to 1\).

\subsection{Multiple Collinear Jumps}

Let all jumps be in the same direction: \(v_{n,k} = v\) for all \(k\). Take \(u_n = v\). Then
\[
\sum_{k=1}^{m_n} \theta_{n,k} (1 - \langle v, v \rangle^2) = 0,
\]
so condition (1) holds exactly. Then
\[
J_n = \left(\sum_{k=1}^{m_n} \theta_{n,k}\right) v v^\top,
\]
and the conditions reduce to the same trace and second moment bounds.

\subsection{Approximately Aligned Jumps with Dominant Magnitude}

Let \(v_{n,1} = e_1\), and for \(k \ge 2\), let \(v_{n,k} = e_1 \cos \phi_{n,k} + e_2 \sin \phi_{n,k}\) with \(\phi_{n,k} \to 0\). Let \(\theta_{n,1} = 1\) and \(\theta_{n,k} = \epsilon_{n,k}\) with \(\sum \epsilon_{n,k} \to 0\). Take \(u_n = e_1\). Then
\[
\sum_{k=1}^{m_n} \theta_{n,k} (1 - \langle e_1, v_{n,k} \rangle^2) = \sum_{k \ge 2} \epsilon_{n,k} (1 - \cos^2 \phi_{n,k}) = \sum_{k \ge 2} \epsilon_{n,k} \sin^2 \phi_{n,k} \to 0,
\]
since \(\epsilon_{n,k} \to 0\). Thus condition (1) holds.

\subsection{Stochastic Example: Poisson Jumps with Aligned Directions}

Let jumps occur according to a Poisson process with intensity \(\lambda_n = n^2\) up to time \(T=1\). Each jump has magnitude \(\theta_k = 1\) and direction \(v_k = e_1 + \xi_k\) where \(\xi_k\) are random with \(\mathbb{E}[\|\xi_k\|^2] = 1/n\). Then:
\[
\operatorname{tr}(J_n) \sim n^2, \quad \mathbb{E}\left[\sum \theta_k (1 - \langle e_1, v_k \rangle^2)\right] = O(n),
\]
so the alignment condition holds in probability. If the diffusion satisfies \(\operatorname{tr}(B_n) = O(n)\), then collapse occurs.

\section{Interpretation of Asymptotic Conditions}

We discuss what the three conditions mean in concrete terms.

\subsection{Condition 1: Weighted Geometric Alignment}

\[
\sum_{k=1}^{m_n} \theta_{n,k} \left(1 - \langle v_{n,k}, u_n \rangle^2 \right) = o\left(\sum_{k=1}^{m_n} \theta_{n,k}\right).
\]

This condition says: the total jump magnitude coming from directions \emph{not} aligned with \(u_n\) is negligible compared to the total jump magnitude. This is the natural mathematical formulation of "most jumps are aligned with a common direction."

\subsection{Condition 2: Trace Dominance}

\[
\operatorname{tr}(B_n) = o(\operatorname{tr}(J_n))
\]

means the total normal variation is negligible compared to total jump variation.

\subsection{Condition 3: Second Moment Dominance}

\[
\operatorname{tr}(B_n^2) = o(\lambda_1(J_n)^2)
\]

implies \(\|B_n\|_{\mathrm{op}} = o(\lambda_1(J_n))\). This is stronger than condition 2 and ensures that the background does not create additional large eigenvalues.

\section{Connection to Previous Work}

This framework unifies and generalizes results from earlier analyses.

\subsection{Determinant Criterion}

Previous work used \(\log \det(Q_n)\) to detect regime changes. For an invertible background \(B_n\), the matrix determinant lemma applied sequentially gives
\[
\det(B_n + J_n) = \det(B_n) \cdot \det(I + B_n^{-1/2} J_n B_n^{-1/2}).
\]
When \(J_n\) is low-rank and aligned, the leading eigenvalue dominates, so determinant growth is governed by the same spectral concentration.

\subsection{Coupling Observable}

In the two-channel setting, the coupling ratio approaches 1 when jumps are aligned across channels. This is a special case of geometric alignment.

\subsection{Effective Rank}

The effective rank collapse to 1 is the strongest diagnostic: it requires no eigenvector computation and is robust to noise.

\subsection{Unification}

Previous papers examined determinant growth, coupling between channels, and ensemble behavior. This two-step characterization shows all three are manifestations of the same underlying principle.

\section{Conclusion}

We have shown that extreme events introduce a low-rank structure into the increment matrix \(Q_n\). The main contribution is a complete characterization of spectral collapse:

\begin{enumerate}
\item \textbf{Geometric alignment induces spectral concentration:} When most of the total jump magnitude comes from directions close to some common vector, the jump matrix satisfies \(\lambda_1(J_n) \sim \operatorname{tr}(J_n)\).

\item \textbf{Spectral concentration induces collapse:} When \(\lambda_1(J_n) \sim \operatorname{tr}(J_n)\) and the background is asymptotically negligible, the full system collapses:
\[
\frac{\lambda_1(Q_n)}{\operatorname{tr}(Q_n)} \to 1, \quad \operatorname{erank}(Q_n) \to 1.
\]

\item \textbf{Collapse forces geometry:} Conversely, if collapse occurs, then necessarily the jump directions are asymptotically aligned and the background is negligible.
\end{enumerate}

We extended these results to the stochastic setting, proving almost sure collapse under natural probabilistic assumptions on the jump process. This framework unifies previous results and provides a practical diagnostic for detecting when a system is controlled by extreme events.


\end{document}